\documentclass{article}
\usepackage[utf8]{inputenc}
\usepackage{amsmath}
\usepackage{amsfonts}
\usepackage{amsthm}
\usepackage{amssymb}

\usepackage{centernot}
\usepackage{url}

\def\mod{\bmod}

\def\intZ{ \mathbb{Z} }
\def\ratQ{ \mathbb{Q} }  
 
\def\natN{  \mathbb{N} }

\numberwithin{equation}{section}

\newtheorem{Theorem}{Theorem}[section]

\newtheorem{Proposition}[Theorem]{Proposition}
\newtheorem{Lemma}[Theorem]{Lemma}
\newtheorem{Observation}{Observation}
\newtheorem{Conjecture}[Theorem]{Conjecture}
\newtheorem*{conjecture*}{Conjecture}

\def\Legendre#1#2{\left( \frac{#1}{#2} \right)}

\newcommand*\conj[1]{\overline{#1}}

\def\rad{{\rm rad}}

\begin{document}

On a conjecture concerning the number of solutions to $a^x+b^y=c^z$, II  
 
     
5 July 2023 

\bigskip

Maohua Le  

Reese Scott  

Robert Styer  

\

\begin{abstract}
Let $a$, $b$, $c$ be distinct primes with $a<b$.  Let $S(a,b,c)$ denote the number of positive integer solutions $(x,y,z)$ of the equation $a^x + b^y = c^z$.  In a previous paper \cite{LeSt} it was shown that if $(a,b,c)$ is a triple of distinct primes for which $S(a,b,c)>1$ and $(a,b,c)$ is not one of the six known such triples then $(a,b,c)$ must be one of three cases.  In the present paper, we eliminate two of these cases (using the special properties of certain continued fractions for one of these cases, and using a result of Dirichlet on quartic residues for the other).  Then we show that the single remaining case requires severe restrictions, including the following: $a=2$, $b \equiv 1 \bmod 48$, $c \equiv 17 \mod 48$, $b > 10^9$, $c > 10^{18}$; at least one of the multiplicative orders $u_c(b)$ or $u_b(c)$ must be odd (where $u_p(n)$ is the least integer $t$ such that $n^t \equiv 1 \bmod p$); 2 must be an octic residue modulo $c$ except for one specific case; $2 \mid v_2(b-1) \le v_2(c-1)$ (where $v_2(n)$ satisfies $2^{v_2(n)} \parallel n$); there must be exactly two solutions $(x_1, y_1, z_1)$ and $(x_2, y_2, z_2)$ with $1 = z_1 < z_2$ and either $x_1 \ge 28$ or $x_2 \ge 88$.  These results support a conjecture put forward in \cite{ScSt6} and improve results in \cite{LeSt}.  
\end{abstract}

2020 Mathematics Subject Classification 11D61


\section{Introduction}  

Let $\mathbb{P}$ be the set of positive rational prime numbers.  We consider $S(a,b,c)$, the number of solutions in positive integers $(x,y,z)$ to the equation
$$ a^x + b^y = c^z,   a, b, c \in \mathbb{P}, a<b.  \eqno{(1.1)}$$
This paper will continue the discussion of a conjecture on (1.1) found in \cite[Conjecture 1.7]{LeSt}:

\begin{conjecture*} 
For $a$, $b$, and $c$ distinct primes with $a<b$, we have $S(a,b,c) \le 1$, except for 

(i)  $S(2,3,5) = 2$, $(x,y,z) = (1,1,1)$ and $(4,2,2)$.  

(ii)  $S(2, 3, 11)=2$, $(x,y,z)=(1,2,1)$ and $(3,1,1)$.

(iii)  $S(2, 5, 3)=2$, $(x,y,z)=(1,2,3)$ and $(2,1,2)$.

(iv)  $S(2, 7, 3)=2$, $(x,y,z)=(1,1,2)$ and $(5,2,4)$.

(v)  $S(3,5,2)=3$, $(x,y,z)=(1,1,3)$, $(1,3,7)$, and $(3,1,5)$.

(vi)  $S(3,13,2)=2$, $(x,y,z)=(1,1,4)$ and $(5,1,8)$.
\end{conjecture*}

There is much previous work on various types of exponential Diophantine equations with prime bases (see, for example, \cite{Alex}, \cite{Be1}, \cite{BrennerFoster}, \cite{Hadano}, \cite{Herschfeld}, \cite{Le1}, \cite{MT}, \cite{N2}, \cite{Sc}, \cite{Sc2}, \cite{ScSt1}, \cite{Uchiyama}, \cite{deW}).  Most such work deals with the familiar Pillai equation $c^z-b^y = a$, taking $c$ and $b$ prime.  In 1985 the first author \cite{Le1} obtained some early results on (1.1) and conjectured that (1.1) has at most one solution in positive integers $(x,y,z)$ with $\min(x,y,z)>1$.   This conjecture is restated in \cite{Le2} and proven in the introduction to \cite{ScSt2}; it is also included in Theorem 1.1 below.  

In \cite{ScSt1} it is shown that the more general equation
$$ (-1)^u p^x + (-1)^v q^y = r^z,  p,q,r \in \mathbb{P}, (p,q,r) \ne (2,2,2), x,y,z, \in \intZ^+, u,v \in \{ 0,1\} $$
has at most two solutions $(x,y,z,u,v)$ except when $(p,q,r)$ is a permutation of one of the following: $(5,3,2)$ which has seven solutions, $(7,3,2)$ which has four solutions, $(11,3,2)$ which has three solutions, $(13,3,2)$ which has three solutions.  But improving this to at most one solution (with listed exceptions) has not been accomplished, even when $(u,v)$ is fixed at $(0,0)$.  

A more recent result is the following, easily derived from \cite[Lemma 1.2, Theorem 1.4, Theorem 1.5, Corollary 1.6, and Theorem 1.8]{LeSt}:

\begin{Theorem} 
If $(a,b,c) = (2,3,5)$, $(2,3,11)$, $(2,5,3)$, $(2,7,3)$, $(3,5,2)$ or $(3,13,2)$, Equation (1.1) has only the solutions given in the above conjecture.  Except for these six cases, if Equation (1.1) has more than one solution, we must have $(a,b,c) = (2,p,q)$ for some odd primes $p>10^9$ and $q>10^{18}$, and there must be exactly two solutions $(x_1,y_1,z_1)$ and $(x_2,y_2,z_2)$ as follows:
$$ 2^{x_1} + p^{y_1} = q, 2 \mid x_1, 2 \mid y_1,  \eqno{(1.2)}$$
 and 
$$2^{x_2} + p^{y_2} = q^{z_2}, 2 \mid x_2, 2 \nmid y_2, 2 \nmid z_2 > 1. \eqno{(1.3)}$$
\end{Theorem}

(1.3) follows from $p \equiv 1 \bmod 3$, shown in \cite{LeSt} using a result of Bennett \cite[Theorem 1.1]{Be2}, and a result of Bauer and Bennett \cite[Corollary 1.7]{BB}.  Using $p \equiv 1 \bmod 3$ leads to the following theorem (Theorem 1.5 of \cite{LeSt}): 

\begin{Theorem} 
If (1.1) has more than one solution and is not one of the six exceptional cases of Theorem 1.1, then we must have $(a,b,c) = (2,p,q)$ for some odd primes $p$ and $q$ satisfying one of the following conditions:

$p \equiv 13 \bmod 24$, $q \equiv 5  \bmod 24$,

$p \equiv 13 \bmod 24$, $q \equiv 17 \bmod 24$,

$p \equiv 1 \bmod 24$, $q \equiv  17 \bmod 24$.

\end{Theorem}

The purpose of this paper is to first eliminate two of the cases in Theorem 1.2 and then show that the single remaining case implies severe restrictions on $p$ and $q$.  We use the following notation: 

If $2^t \parallel n$, we write $v_2(n) = t$.

If $t$ is the least positive integer such that $n^t \equiv 1 \bmod p$ for some prime $p$, we write $u_p(n) = t$.  

We prove the following improvement on Theorem 1.2:

\begin{Theorem}  
If (1.1) has more than one solution and is not one of the six exceptional cases of Theorem 1.1, then we must have $(a,b,c) = (2,p,q)$ for some odd primes $p$ and $q$ satisfying all of the following conditions:

(i):  $p \equiv 1 \bmod 48$, $q \equiv 17 \bmod 48$,  $2 \mid v_2(p-1) \le v_2(q-1)$.   

(ii):  At least one of the multiplicative orders $u_p(q)$ and $u_q(p)$ must be odd.  

(iii):  2 is an octic residue modulo $q$, that is, $2$ is congruent to an eighth power modulo $q$, except when $v_2(p-1) = v_2(q-1) = 4$.  
\end{Theorem}

A further restriction is given by the following: 

\begin{Theorem} 
In Equations (1.2) and (1.3),  the following must hold:

(i): Either $x_1 \ge 28$ or $x_2 \ge 88$.  

(ii):  If $27 \ge v_2(p-1) = v_2(q-1)$, then $x_2 \ge 88$; and if $87 \ge v_2(p-1)$ and $v_2(p-1) <  v_2(q-1)$, then $ x_1 \ge 28$. 

\end{Theorem}  

Theorems 1.3 and 1.4 combine with Theorem 1.1 to give information on cases (other than the six known cases) in which (1.1) might have more than one solution.  We hope this new information might eventually lead to a proof of the above conjecture.   

This conjecture is a special case of the following conjecture given in \cite{ScSt6}:

\begin{Conjecture}  
Let $N(a,b,c)$ be the number of solutions in positive integers $(x,y,z)$ to the equation 
$$a^x + b^y = c^z, a,b,c \in \intZ^+, b>a>1, \gcd(a,b) = 1, \eqno{(1.4)}$$
with $a$, $b$, $c$ not perfect powers.  
 
If $N(a,b,c)>1$, then $(a,b,c)$ must be one of the following:

(i)  $N(2, 2^r-1, 2^r+1)=2$, $(x,y,z)=(1,1,1)$ and $(r+2, 2, 2)$, where $r$ is a positive integer with $r \ge 2$, $r \ne 3$.  

(ii)  $N(2, 3, 11)=2$, $(x,y,z)=(1,2,1)$ and $(3,1,1)$.

(iii)  $N(2, 3, 35)=2$, $(x,y,z)=(3,3,1)$ and $(5,1,1)$.

(iv)  $N(2, 3, 259)=2$, $(x,y,z)=(4,5,1)$ and $(8,1,1)$.

(v)  $N(2, 5, 3)=2$, $(x,y,z)=(1,2,3)$ and $(2,1,2)$.

(vi)  $N(2, 5, 133)=2$, $(x,y,z)=(3,3,1)$ and $(7,1,1)$.

(vii)  $N(2, 7, 3)=2$, $(x,y,z)=(1,1,2)$ and $(5,2,4)$.

(viii)  $N(2, 89, 91)=2$, $(x,y,z)=(1,1,1)$ and $(13,1,2)$.

(ix)  $N(2, 91, 8283)=2$, $(x,y,z)=(1,2,1)$ and $(13,1,1)$.

(x)  $N(3,5,2)=3$, $(x,y,z)=(1,1,3)$, $(1,3,7)$, and $(3,1,5)$.

(xi)  $N(3,10,13)=2$, $(x,y,z)=(1,1,1)$ and $(7,1,3)$.

(xii)  $N(3,13,2)=2$, $(x,y,z)=(1,1,4)$ and $(5,1,8)$.

(xiii)  $N(3, 13, 2200)=2$, $(x,y,z)=(1,3,1)$ and $(7,1,1)$.

\end{Conjecture}

An effective upper bound for $N(a,b,c)$ was first given by A. O. Gel'fond \cite{G} (Mahler \cite{M} had earlier shown that the number of solutions was finite, using his $p$-adic analogue of the Diophantine approximation method of Thue-Siegel, but his method is ineffective).  A straightforward application of an upper bound on the number of solutions of binary $S$-unit equations due to F. Beukers and H. P. Schlickewei \cite{BSch} gives $N(a,b,c) \le 2^{36}$.  The following more accurate upper bounds for $N(a,b,c)$ have been obtained in recent years:

(i)  (R. Scott and R. Styer \cite{ScSt6})  If $ 2 \nmid c$ then $N(a,b,c) \le 2$.

(ii)  (Y. Z. Hu and M. H. Le \cite{HL1})  If $\max\{ a,b,c\} > 5 \cdot 10^{27}$, then $N(a,b,c) \le 3$.  

(iii)   (Y. Z. Hu and M. H. Le \cite{HL2})  If $2 \mid c$ and $\max\{ a,b,c\} > 10^{62}$, then $N(a,b,c) \le 2$.

(iv)  (T. Miyazaki and I. Pink \cite{MP})  If $2 \mid c$ and $\max\{a,b,c \} \le 10^{62}$, then $N(a,b,c) \le 2$ except for $N(3,5,2) = 3$.  

More recently, Miyazaki and Pink \cite{MP2} have begun work on improving $N(a,b,\allowbreak c) \allowbreak  \le 2$ to $N(a,b,\allowbreak c) \le 1$ under certain conditions, including some specific results such as $c \not\in \{ 2, 3, 5, 6, 17, 257, 65537 \}$ when $N(a,b,c)>1$ ($c \ne 6$ has not previously been shown even for the more specific Pillai equation mentioned above).  

Nevertheless, the problem of establishing $N(a,b,c) \le 1$ with a finite number of specified exceptions remains open.  This open question is addressed by the above Conjecture 1.5.  

In Sections 3 and 4 we show that the first two cases given in Theorem 1.2 are impossible, and then, in Section 5, we prove Theorems 1.3 and 1.4. In Section 6, we consider (1.3) in the context of the $abc$ conjecture.  

\section{Preliminary Lemmas}  

\begin{Lemma}[Theorem 6 of \cite{ScSt1}]  
Let $p$, $q$ be distinct odd positive primes.  For a given positive integer $k$, the equation
$$ q^n - p^m = 2^k, m, n \in \natN, $$
has at most one solution in positive integers $(m,n)$.  
\end{Lemma}

\begin{Lemma}[Theorem 1.1 of \cite{Be2}]  
Let $c$ and $b$ be positive integers.  Then there exists at most one pair of positive integers $(z,y)$ for which 
$$ 0 < |c^{z} - b^y| < \frac{1}{4} \max\{ c^{z/2}, b^{y/2} \}.$$
\end{Lemma}

\begin{Lemma}  
If a prime $p$ is of the form $a^2 + 64 b^2$ for some integers $a$ and $b$, then 2 is a quartic residue modulo $p$.  
\end{Lemma}

\begin{proof}
A proof is found in \cite{Dir} which is simpler than Gauss's earlier proof of a conjecture of Euler.  
\end{proof}

\begin{Lemma}  
For any prime $p \equiv 1 \bmod 16$, 2 is an octic residue modulo $p$ if and only if $p = a^2 + 256 b^2$ for some integers $a$ and $b$.  
\end{Lemma}  

\begin{proof}   
This lemma is found in Whiteman \cite{Wh}, who cites Reuschle \cite{Re} for the original statement of the lemma and Western \cite{We} for the first proof.  
\end{proof}

\begin{Lemma}[Theorem 1.8 of \cite{LeSt}]   
If (1.1) has more than one solution and is not one of the six exceptional cases given in Theorem 1.1, then $a=2$, $b>10^9$, and $c > 10^{18}$.
\end{Lemma}

\section{$(p,q) \not\equiv (13, 5) \bmod 24$}   

The purpose of this section is to prove the following

\begin{Proposition}  
If $p \equiv q \equiv 5 \bmod 8$, then the equation
$$ 2^x + p^y = q^z, p, q \in \mathbb{P}$$
has at most one solution in positive integers $(x,y,z)$.  
\end{Proposition}

We will use four lemmas.

\begin{Lemma}  
Let $D$ be a natural number which is not a perfect square. Let $h$, $k$, $h_1$, and $k_1$ be integers. Suppose the equation $h^2 - D k^2 = -1$ is solvable, and that $h_1 + k_1 \sqrt{D}$ is its fundamental solution.  
Let $p$ be any prime dividing $h_1$.  Then if $U^2 - V^2 D = 1$, we must have $p \mid V$.    
\end{Lemma} 

\begin{proof}
The lemma follows from Theorem 106 of \cite{N} and Lemma 1 of \cite{Sc}. 
\end{proof}

For the next two lemmas we establish notation for the continued fraction for $\sqrt{D}$ and its convergents (the basic results on which this notation is based can be found in \cite{Pe}).  For any non-square positive integer $D$ let
$$ \sqrt{D} = [ a_0, \overline{ a_1, \dots, a_s }]$$
represent the continued fraction expansion of $\sqrt{D}$.  
Let $\frac{P_m}{Q_m}$ be the $m$-th convergent of $\sqrt{D}$ and let 
$$ k_m = (-1)^{m+1} ( P_m^2 - Q_m^2 D),  \eqno{(3.1)}$$
noting that (as shown in \cite{Pe}) all the $k_m$ are positive integers with
$$ k_{ns + j} = k_j, j = 0, \dots, s-1, n \in \intZ^+. \eqno{(3.2)}$$
We are now ready to state

\begin{Lemma}[Theorem 10.8.2 of \cite{Pe}] 
Let $k$ be an integer.  If $|k| < \sqrt{D}$ and $(x,y)$ is a solution of $x^2 - y^2 D = k$ with $\gcd(x, yD)=1$, then $\frac{|x|}{|y|} $ is a convergent of the continued fraction for $\sqrt{D}$. 
\end{Lemma} 

\begin{Lemma}    
If $| x^2 - y^2 D| < \sqrt{D}$ and $\gcd(x, yD)=1$, then $| x^2 - y^2 D| = k_m $ for some $m \le s-1$.
\end{Lemma} 

\begin{proof} 
By Lemma 3.3, $x/y$ is a convergent of the continued fraction for $[a_0, \allowbreak  \overline{  a_1,  \allowbreak   \dots, a_s }]$, so that (3.2) gives the lemma.  
\end{proof}

The following lemma is obtained by direct calculation.  

\begin{Lemma}  
If $D = p^{2n} +4$ where $p, n \in \intZ^+$ with $2 \nmid p$, then we have 
$$ \sqrt{D} = [ p^n, \overline{  (p^{n} -1)/2, 1, 1 , (p^{n} -1)/2, 2 p^n }]. \eqno{(3.3)}$$
$$
{
\begin{split} 
\frac{P_0}{Q_0} = \frac{p^n}{1},  \frac{P_1}{Q_1}  =& \frac{(p^{2n} - p^n + 2)/2 }{(p^n-1)/2},  \frac{P_2}{Q_2} = \frac{(p^{2n} + p^n + 2)/2 }{(p^n+1)/2} ,  \\ 
 &\frac{P_3}{Q_3} = \frac{p^{2n}+2 }{p^n}, \frac{P_4}{Q_4} = \frac{(p^{2n} +3)p^n/2 }{(p^{2n}+1)/2}. 
\end{split}
}
\eqno{(3.4)}
$$
$$ k_0 = 4, k_1 = p^n, k_2 = p^n, k_3 = 4, k_4 = 1.  \eqno{(3.5)}$$

\end{Lemma}

We are now ready to give 

\begin{proof}[Proof of Proposition 3.1]
Let $p, q \in \mathbb{P}$, $ 3 \nmid pq$. Assume $2^x + p^y = q^z$ has two solutions $(x, y, z)$.  
Let $Z_1$ be the least positive integer such that there exist rational integers $X$ and $Y$ with $\gcd(X,Yq)=1$ satisfying $X^2 - Y^2 q = \pm p^{Z_1}$ (such $Z_1$ exists by Theorem 1.1).  
Let $\theta$ be any integer of the field $\ratQ(\sqrt{q})$ with norm $-p^{Z_1}$ such that $p \nmid \theta$ and $\theta$ has rational integer coefficients (such $\theta$ exist by Theorem 107 of \cite{N} and Theorem 1.2). 
By Lemma 3.1 of \cite{Sc2} $Z_1 \mid Z$ in any solution of $X^2 - Y^2 q = \pm p^Z$ with $X$ and $Y$ rational integers, $\gcd(X,Yq)=1$.  Applying Theorem 1.1 and using (1.3), let $\beta = 2^{x_2/2}+q^{(z_2-1)/2} \sqrt{q}$.  By (1.3), $\beta$ has norm $-p^{y_2}$, so $Z_1 \mid y_2$. Let $\alpha = \theta^t$ where $t = \frac{y_2}{Z_1}$, where $y_2$ is as in (1.3).  Since $y_2$ is odd, $t$ is odd, 
so that $\alpha$ has norm $-p^{y_2}$.  Now we have $\beta \conj{\beta} = \alpha \conj{\alpha}$.  Noting $\Legendre{q}{p} = 1$ where $\Legendre{q}{p}$ is the Legendre symbol, let $\mathfrak{p} \conj{\mathfrak{p}}$ be the unique ideal factorization of $p$ in $\ratQ(\sqrt{q})$.  Now we have the equation in ideals 
$$[\beta] [\conj{\beta}] = [\alpha] [\conj{\alpha}] = \mathfrak{p}^{y_2} \conj{\mathfrak{p}}^{y_2}$$ 
where $p \nmid \alpha$ and $p \nmid \beta$, so that $[\beta] = \mathfrak{p}^{y_2}$ or $[\beta] = \conj{\mathfrak{p}}^{y_2}$, and $[\alpha] = {\mathfrak{p}}^{y_2}$ 
or $[\alpha] = \conj{\mathfrak{p}}^{y_2}$.  Thus, $[\beta]= [\alpha]$ or $[\beta] = [\conj{\alpha}]$  so there exists a unit $\delta$ such that 
$$\beta = \delta \alpha, {\rm \ or  \ }  \beta = \delta \conj{\alpha}.$$
Since the norms of $\alpha$ and $\beta$ are odd, $\delta$ must have rational integer coefficients.  Let $\xi = \theta$ or $\conj{\theta}$ according as $\beta = \delta \alpha$ or $\delta \conj{\alpha}$.  
Thus we have 
$$ 2^{x_2/2}+q^{(z_2-1)/2} \sqrt{q} = \xi^t \delta,  \eqno{(3.6)}$$
where $\delta$ has norm 1 since $\xi^t$ has norm $-p^{y_2}$.  
Let
$$ \theta = X_1 + Y_1 \sqrt{q}, \delta = U+V \sqrt{q},  \eqno{(3.7)}$$
where $X_1$, $Y_1$, $U$, and $V$ are nonzero rational integers.  

Now we assume $q \equiv 5 \bmod 8$ and apply Lemma 3.5 with $D=q$, $n = y_1/2$, and $p$ as in (1.2), noting that, since $q \equiv 5 \bmod 8$ and $2 \mid y_1$, we must have $x_1=2$ in (1.2).  By (3.1), (3.4), and (3.5) we see that 
$$ P_2^2 - q  Q_2^2  = - p^{y_1/2}.  \eqno{(3.8)} $$
Suppose $X^2 - Y^2 q = \pm p^{n_0} $ where $\gcd(X, Y) =1$ and $n_0 < y_1/2$.  Then by Lemma 3.4 we must have $p^{n_0} = k_m$ for some $m \le 4$, contradicting (3.5).  So we have 
$$ Z_1 = y_1/2.  \eqno{(3.9)}$$
Since $\theta$ is any integer of the field with rational integer coefficients having norm $-p^{Z_1}$ and satisfying $p \nmid \theta$, by (3.8) and (3.4) we can take $\theta = X_1 + Y_1 \sqrt{q}$ where 
$$ (X_1, Y_1, Z_1) = ( \frac{1}{2} (p^{y_1} + p^{y_1/2} + 2), \frac{1}{2} (p^{y_1/2} + 1), y_1/2)  \eqno{(3.10)}$$
From (3.10) we see that 
$$ X_1^2 \equiv Y_1^2 q \equiv 1 \bmod p,  \eqno{(3.11)}$$
where the last congruence holds since $q \equiv 2^{x_1} = 4 \bmod p$, noting $x_1=2$ in (1.2). 

By (3.4) and (3.5), we see that $P_4 + Q_4 \sqrt{q} $ is the fundamental solution of $x^2 - y^2 q = -1$ (since $P_m$ and $Q_m$ increase with $m$).  
Since $p \mid P_4$, by Lemma 3.2 we must have $p \mid V$ in (3.7).  Also $U^2 = V^2 q + 1 \equiv 1 \bmod p$.  So, in (3.7), we have
$$ V \equiv 0 \bmod p, U \equiv \pm 1 \bmod p.  \eqno{(3.12)}$$
By (3.6) we have 
$$ 2^{x_2/2}+q^{(z_2-1)/2} \sqrt{q} = (X_1 \pm Y_1 \sqrt{q})^t (U+V \sqrt{q}). \eqno{(3.13)}$$
Using (3.11) and considering the binomial expansion of $(X_1 \pm Y_1 \sqrt{q})^t = X_t + Y_t \sqrt{q} $ with $t$ odd we find
$$ \pm Y_t \equiv  2^{t-1} Y_1 = 2^{t-2} (2 Y_1) \equiv 2^{t-2} \bmod p,  \eqno{(3.14)}$$
where the last congruence follows from (3.10). 
Since $q \equiv 4 \bmod p$, (3.13) and (3.12) give 
$$ 2^{z_2 -1} \equiv q^{(z_2-1)/2 } \equiv X_t V + Y_t U \equiv \pm Y_t \bmod p.  \eqno{(3.15)}$$
(3.14) and (3.15) give
$$ 2^{z_2 - 1} \equiv \pm 2^{t-2} \bmod p. \eqno{(3.16)}$$
Since $z_2$ and $t$ are both odd, and $p \equiv 1 \bmod 4$ by Theorem 1.2, (3.16) requires $\Legendre{2}{p} = 1$,  $p \not\equiv 5 \bmod 8$, completing the proof of Proposition 3.1.
\end{proof}

\section{$(p,q) \not\equiv (13, 17) \bmod 24$} 

The purpose of this section is to prove the following: 

\begin{Proposition} 
If $p \equiv 5 \bmod 8$ and $q \equiv 1 \bmod 8$, 
then the equation
$$ 2^x + p^y = q^z, p, q \in \mathbb{P}$$
has at most one solution in positive integers $(x,y,z)$. 
\end{Proposition}

We first prove a general lemma.  
We use the following notation: let $d$ be a primitive root of a prime $p$; if an integer $n \equiv d^i \bmod p$ with $0 < i \le p-1$, we call $i$ the index of $n$ for that primitive root $d$ and write 
$$ i = i_p(n). $$
We also use the notation $v_2(n)$ to indicate the greatest integer $t$ such that $2^t \mid n$.  Notice that $v_2( \gcd(i_p(n), p-1))$ is independent of the choice of primitive root $d$.  For brevity, we use the following notation: 
$$ w_p(n) =  \min( v_2(i_p(n)), v_2(p-1) ),$$
so that $0 \le w_p(n) \le v_2(p-1)$.    

We use three simple observations:

\begin{Observation} 
If $a \equiv b \bmod p$, then $w_p(a) = w_p(b)$.  
\end{Observation}

\begin{Observation} 
We have 

(i)  $w_p(-a) = w_p(a)$ when $w_p(a) < v_2((p-1)/2)$,

(ii)  $w_p(-a) = v_2(p-1)$ when $w_p(a) = v_2((p-1)/2)$, 

(iii) $w_p(-a) = v_2((p-1)/2)$ when $w_p(a) = v_2(p-1)$.

\end{Observation}

\begin{Observation} 
For a given prime $p$ and a given integer $a$, if $w_p(a^t) < v_2(p-1)$, then $v_2(t)$ is determined by $w_p(a^t)$.
\end{Observation} 

\begin{Lemma} 
If (1.2) and (1.3) hold with $v_2(x_1) = v_2(x_2)$, then
$$w_q(2^{x_1}) = v_2((q-1)/2). \eqno{(4.1)}$$ 
\end{Lemma} 

\begin{proof} 
We consider the solutions $(x_1, y_1,z_1)$ and $(x_2, y_2, z_2)$ modulo $q$.  

If $w_q(2^{x_1}) < v_2((q-1)/2)$, then, by Observation 2, $w_q(p^{y_1}) = w_q(2^{x_1})$.  But then also, since $v_2(x_2) = v_2(x_1)$, we have $w_q(2^{x_2}) = w_q( 2^{x_1}) = w_q(p^{y_1})$ and, by Observation 2, $w_q(2^{x_2}) = w_q(p^{y_2})$ so that $w_q(p^{y_1}) = w_q(p^{y_2})$.  But this is impossible by Observation 3 since $2 \nmid y_1 - y_2$ and $w_q(2^{x_1}) = w_q(p^{y_1}) = w_q(p^{y_2}) < v_2((q-1)/2)$.  

Similarly, if $w_q(2^{x_1}) > v_2((q-1)/2)$, then, by Observation 2, $w_q(p^{y_1}) = v_2((q-1)/2)$, and, since $w_q(2^{x_1}) = w_q(2^{x_2})$, also $w_q(p^{y_2})= v_2((q-1)/2)$, which is impossible by Observation 3 since $2 \nmid y_1 - y_2$.  

So we must have (4.1).
\end{proof}

\begin{proof}[Proof of Proposition 4.1]
Assume (1.1) has more than one solution with $p \equiv 5 \bmod 8$, $q \equiv 1 \bmod 8$.  
By Theorem 1.2, we can assume that $ 5 \nmid pq$.  Considering congruences modulo 8, we find $2^{x_2} = 4$.  Since 2 is a quadratic nonresidue of $p$, we have $w_p(2)=0$, so that, by Observation 1, we have $1 = w_p(4) = w_p(2^{x_2}) = w_p(q^{z_2}) = w_p(q) = w_p(2^{x_1})$, so that, by Observation 3, 
$$v_2(x_1) = v_2(x_2) = 1, \eqno{(4.2)} $$
noting that $w_p(q^{z_2}) = w_p(q)$ since $z_2$ is odd.  So we can apply Lemma 4.2 to obtain (4.1).  

From (4.2) we have $v_2(x_1) = 1$, so that $2^{x_1} \equiv 4 \bmod 5$.  Since $5 \nmid pq$ and $2 \mid y_1$, we must have $p^{y_1} \equiv 4 \bmod 5$.  So $2 \parallel y_1$, giving 
$$ p^{y_1} \equiv 9 \bmod 16.  \eqno{(4.3)}$$
Since $q \equiv 1 \bmod 8$, $2^{x_1} > 4$, which, along with (4.2), gives 
$$ 2 \parallel x_1 > 2. \eqno{(4.4)}$$
Considering congruences modulo 16 and using (1.2), (4.3), and (4.4), we see that $2^3 \parallel q-1$, so that $v_2( (q-1)/2 )=2$, and (4.1) becomes
$$ w_q(2^{x_1}) = v_2((q-1)/2) =2. \eqno{(4.5)}$$ 
 
From (4.4) we see that $q = p^{y_1} + 2^{x_1}$ is of the form $a^2 + 64 b^2$, so, by Lemma 2.3, $2$ is a quartic residue modulo $q$, so that $4 \mid i_q(2)$ for any choice of primitive root $d$.  
Thus $8 \mid i_q(2^{x_1})$ so that 
$$w_q(2^{x_1}) = \min(v_2(i_q(2^{x_1})), v_2(q-1)) = 3,$$
contradicting (4.5), proving Proposition 4.1.  
\end{proof}

\section{Proofs of Theorem 1.3 and Theorem 1.4}  

\begin{proof}[Proof of Theorem 1.3]
From Theorem 1.2, Proposition 3.1, and Proposition 4.1, we have 
$$ (p,q) \equiv (1,17) \bmod 24. \eqno{(5.1)}$$
We prove (i), (ii), and (iii) separately.

(i): \quad  
From (1.2) and (1.3) we have 
$$ 2^{x_1} + (p^{y_1} - 1) = (q-1), 2 \mid x_1, 2 \mid y_1,  $$
and
$$ 2^{x_2} + ( p^{y_2} - 1) = (q^{z_2} - 1), 2 \mid x_2, 2 \nmid y_2, 2 \nmid z_2. $$
If $v_2(p-1) > v_2(q-1)$, then we see that $x_1= v_2(q-1) = v_2(q^{z_2} - 1) = x_2$, contradicting Lemma 2.1, so that
$$v_2(p-1) \le v_2(q-1).  \eqno{(5.2)} $$
So now we have either  
$$ v_2(p-1) = v_2(q-1) = x_1 \eqno{(5.3)}$$
or  
$$ v_2(q-1) > v_2(p-1) = x_2.  \eqno{(5.4)}$$
Since $2 \mid x_1$ and $2 \mid x_2$, from (5.3) and (5.4) we have  
$$  2 \mid v_2(p-1). \eqno{(5.5)}$$
From (5.1) and (5.5) we have
$$ v_2(p-1) \ge 4, \eqno{(5.6)}$$
which, in combination with (5.1) and (5.2), gives
$$ p \equiv 1 \bmod 48, q \equiv 17 \bmod 48. \eqno{(5.7)}$$
 (5.7), (5.5), and (5.2) give (i) of Theorem 1.3.

(ii): \quad 
We use the notation of Sections 1 and 4.  Using Observation 1 of Section 4 and noting that $z_2$ is odd, we find
$$ w_p(2^{x_1}) = w_p(q) = w_p(q^{z_2}) = w_p(2^{x_2}).$$
If $u_p(q) $ is even, then $w_p(q) \le v_2(\frac{p-1}{2})$, so that $w_p(2^{x_1}) = w_p(2^{x_2}) \le v_2( \frac{p-1}{2})$, so that, by Observation 3 of Section 4, 
$$v_2(x_1) = v_2(x_2). \eqno{(5.8)}$$
So we can apply Lemma 4.2 to find 
$$ w_q(2^{x_1}) = v_2(\frac{q-1}{2}).  $$
So also, by (5.8), $w_q(2^{x_2}) = v_2(\frac{q-1}{2})$, so that by (ii) of Observation 2 of Section 4, $w_q(p^{y_2} ) = w_q(p) = v_2(p-1)$, so that $u_q(p)$ is odd, giving (ii) of Theorem 1.3.    

(iii): \quad 
Assume that we do not have $v_2(p-1) = v_2(q-1)=4$.  We can also assume we do not have $v_2(p-1) = v_2(q-1) = 6$ since then (1.2) becomes $2^6 + p^{y_1} = q$ so that, by Lemmas 2.3 and 2.4, $w_q(2) = 2$ and $w_q(2^{x_1}) = w_q(2)+v_2(6) = 3$; applying Observation 2(i) to equations (1.2) and (1.3) and noting that $v_2(y_1) > v_2(y_2)$, we find $3 = w_q(p^{y_1}) > w_q(p^{y_2}) = w_q(2^{x_2})$, requiring $ 2 \nmid x_2$, contradicting (1.3).  

So now, if $v_2(p-1) = v_2(q-1)$, $x_1 \ge 8$ (by (5.5)).  And if $v_2(p-1) < v_2(q-1)$, then $x_2 = v_2(p-1)$ and $x_1 > v_2(p-1)$, so that, by (5.5) and (5.6), $x_1 \ge 8$, unless $x_2=4$ and $x_1 = 6$ which is an impossible case by Lemmas 2.2 and 2.5.  So $x_1 \ge 8$, so that $q$ is of the form $a^2 + 256 b^2$.  By (5.7), $q \equiv 1 \bmod 16$. So (iii) of Theorem 1.3 follows from Lemma 2.4.  
\end{proof}

\begin{proof}[Proof of Theorem 1.4]
By Corollary 1.6 of \cite{LeSt}, $z_2 > 1$ in (1.3).  So, using Lemma 2.5, we have 
$$ 2^{27} < \frac{10^{9}}{4} < \frac{q^{1/2}}{4}, 2^{87} < \frac{10^{27}}{4} <  \frac{q^{3/2}}{4} \le \frac{q^{z_2/2}}{4}, $$
so that Lemma 2.2 applies to give (i) of Theorem 1.4.  So (5.3) and (5.4) give (ii) of Theorem 1.4.  
\end{proof}

\section{Unlikelihood of Equation (1.3)}  

In this section we consider the equation (1.3) in the context of the $abc$ conjecture.  Let $a$, $b$, and $c$ be positive integers such that $a+b=c$; define $Q = Q(a,b,c)$ as 
$$ Q = \frac{\log(c)}{\log(\rad(abc))}$$
where $\rad(m)$ is the product of all distinct primes dividing $m$. Then for (1.3) we have 
\begin{align*} 
Q =  \frac{z_2 \log(q)}{\log(2)+ \log(p) + \log(q)}  & \ge \frac{3 \log(q) }{ (3/2) \log(q)+ \log(2)}  \\
 & =   2 - \frac{ 2 \log(2)}{(3/2) \log(q) + \log(2)} > 1.97  
\end{align*}
by Lemma 2.5.  

The highest value for $Q$ found in recent researches on the $abc$ conjecture is $Q = 1.62991$ for $(a,b,c) = (2, 3^{10} \cdot 109, 23^5)$.  
If $z_2 > 3$, then we have $Q > 3.29$: if a conjecture of Tenenbaum (quoted in Section B19 of \cite{Guy}) is true, then $Q = 3.29$ is impossible, so that $z_2=3$.  

If $y_2>1$, then $\min(x_2, y_2, z_2) > 2$, so that (1.3) contradicts the familiar Beal conjecture (see Section B19 of \cite{Guy}).  If we assume the Beal conjecture is true, then we can eliminate the case $v_2(p-1) < v_2(q-1)$ from consideration, since this case requires $y_2 > y_1 \ge 2$ (if $v_2(p-1) < v_2(q-1)$, then $x_2 = v_2(p^{y_2}-1) < \min(v_2(p^{y_1}-1), v_2(q-1)) \le x_1$, so that, since $q^{z_2} > q^{z_1} = q$, we must have $y_2>y_1 \ge 2$).

If $y_2 = 1$, we can assume $v_2(p-1) = v_2(q-1)$.  Writing (1.3) as  
$$ 2^{x_2} - q^{z_2} = (2^{x_2/2})^2 - (q^{(z_2-1)/2})^2 q = -p, |-p| < \sqrt{q},$$
we see that $\frac{2^{x_2/2}}{q^{(z_2-1)/2}}$ must be a convergent of the continued fraction expansion of $\sqrt{q}$. If it can be shown that $y_2 = 1$ is impossible, then the conjecture at the beginning of this paper would follow from the Beal conjecture.


\begin{thebibliography}{1}

\bibitem{Alex} 
L. J. Alex, Diophantine equations related to finite groups, {\it Comm. Algebra}, {\bf 4}, no. 1 (1976), 77--100.

\bibitem{BB}
  M. Bauer and M. Bennett, Applications of the hypergeometric method to the generalized Ramanujan-Nagell equation, {\it Ramanujan J.}, {\bf 6} (2002), 209--270.  

\bibitem{Be1}
   M. Bennett, On some exponential equations of S. S. Pillai, {\it Canadian Journal of Mathematics}, {\bf 53} no. 5 (2001), 897--922.


\bibitem{Be2} 
M. A. Bennett, Differences between perfect powers, {\it Canad. Math. Bull.}, {\bf 51} (2008), 337--347.

\bibitem{BSch} 
F. Beukers and H.P. Schlickewei,
The equation $x + y = 1$ in finitely generated groups,
{\it Acta Arith.},
{\bf 78} (1996),
189---199.


\bibitem{BrennerFoster}
 J. L. Brenner and L. L. Foster, Exponential Diophantine Equations, {\it Pacific J. of Math.}, {\bf 101} no. 2 (1982), 263--301. 

\bibitem{Dir}
G. L. Dirichlet, Ueber den biquadratischen Character der Zahl \lq\lq Zwei,'', {\it J. Reine Angew.
Math.}, {\bf 57} (1860), 187--188; or {\it Werke}, {\bf 2} (1897), 261--262.


\bibitem{G}
A.~O. {Gel'fond},
\newblock Sur la divisibilit\'e de la diff\'erence des puissance de deux
  nombres entiers par une puissance d'un id\'eal premier,
\newblock {\em Mat. Sb.}, {\bf 7} (49) (1940), 7--25.


\bibitem{Guy} Guy, Richard. 
{\em Unsolved Problems in Number Theory}. Third edition. New York:  Springer Verlag. 2004.

\bibitem{Hadano}
T. Hadano, On the Diophantine equation $a^x+b^y=c^z$, {\it Math. J. Okayama Univ.}, {\bf 19}, no. 1 (1976/77), 25--29.

\bibitem{Herschfeld} 
A. Herschfeld, The equation $2^x-3^y = d$, {\it Bull. Amer. Math.  Soc. } (N.S.), {\bf 42} no. 4 (1936), 231--234.


\bibitem{HL1}
Y.-Z. Hu and M.-H. Le,
\newblock An upper bound for the number of solutions of ternary purely
  exponential diophantine equations,
\newblock {\em J. Number Theory}, {\bf 183} (2018), 62--73.

\bibitem{HL2} 
Y.-Z. Hu and M.-H. Le, 
An upper bound for the number of solutions of ternary purely exponential diophantine equations II, 
{\it Publ. Math. Debrecen} {\bf 95} (2019),
335--354.


\bibitem{Le1} 
M. H. Le, On the Diophantine equation $a^x + b^y = c^z$, {\it J. Changchun Teachers College Ser. Nat. Sci.} {\bf 2} (1985), 50--62 (in Chinese).


\bibitem{Le2} 
M. H. Le, A conjecture concerning the exponential Diophantine equation $a^x + b^y = c^z$, {\it Acta Arithmetica} {\bf 106.4} (2003), 345--353. 


\bibitem{LeSt} 
\newblock  M. H. Le and R. Styer,  
\newblock  On a conjecture concerning the number of solutions to $a^x+b^y=c^z$,
\newblock Bulletin of the Australian Math. Soc.,
\newblock 2022  


\bibitem{M}
K.~Mahler,
\newblock {Zur Approximation algebraischer Zahlen I: \"Uber den gr\"ossten
  Primteiler bin\"arer Formen},
\newblock {\em Math. Ann.}, {\bf 107} (1933), 691--730.


\bibitem{MP}
T. Miyazaki and I. Pink, Number of solutions to a special type of unit equation in two variables, 2020, to appear in Amer. J. Math.


\bibitem{MP2}
T. Miyazaki and I. Pink, Number of solutions to a special type of unit equation in two variables II, 2022, arXiv:2205.11217.


\bibitem{MT}
 D. Z. Mo and R. Tijdeman, Exponential Diophantine equations with four terms, {\it Indag. Math. (N.S.)}, {\bf 3} no. 1 (1992), 47--57.


\bibitem{N}
T. Nagell, {\it Introduction to Number Theory}, Wiley, 1951.   


\bibitem{N2}
 T. Nagell, Sur une classe d'equations exponentielles, {\it Ark. Mat.}, {\bf 3} (1958), 569--582.

\bibitem{Pe}
O. Perron, {\it Die Lehre von den Kettenbr\"uche}, Chelsea, 1950. 


\bibitem{Sc}     
R. Scott, 
\newblock On the Equations $p^x-b^y = c$ and $a^x+b^y=c^z$, 
\newblock {\em Journal of Number Theory}, {\bf 44} no. 2 (1993), 153--165. 


\bibitem{Sc2} 
R. Scott, Elementary treatment of $p^a \pm p^b + 1=x^2$,  arxiv:math.0608796.

\bibitem{ScSt1}    
R.~Scott and R.~Styer,
\newblock On $p^x - q^y = c$ and related three term exponential diophantine
  equations with prime bases,
\newblock {\em J. Number Theory}, {\bf 105} no. 2 (2004), 212--234.


\bibitem{ScSt2}
R. Scott and R. Styer, 
\newblock On the generalized Pillai equation $\pm a^x \pm b^y = c$, 
\newblock {\it Journal of Number Theory}, {\bf 118} (2006), 236--265.


\bibitem{ScSt6}     
R. Scott and R. Styer, 
Number of solutions to $a^x + b^y = c^z$, 
{\em Publ. Math. Debrecen}
{\bf 88}
(2016),
 131--138.


\bibitem{Re}
C. G. Reuschle, Mathematische Abhandlung, enthaltend neue Zahlentheoretische Tabellen,
Programm zum Schlusse des Schuljahrs 1855--56 am K\"oniglichen Gymnasium zu
Stuttgart, (1856), 61 pp. 


\bibitem{Uchiyama}
S. Uchiyama, On the Diophantine equation $2^x = 3^y + 13^z$, {\it Math. J. Okayama Univ.}, {\bf 19} no. 1 (1976/77), 31--38.  

\bibitem{deW}
B. M. M. de Weger, Solving exponential Diophantine equations using lattice basis reduction algorithms, {\it J. Number Theory}, {\bf 26} no. 3 (1987), 325--367.


\bibitem{We}
A. E. Western, Some criteria for the residues of eighth and other powers, {\it Proc. London Math.
Soc.}, {\bf 9} no. 2 (1911), 244--272. 


\bibitem{Wh} 
A. L. Whiteman, The Sixteenth Power Residue Character of 2, 
{\it Canadian Journal of Mathematics}, {\bf 6}  (1954), 364--373.   


\end{thebibliography}
\end{document}